\documentclass[reqno,12pt]{amsart}

\usepackage{latexsym}
\usepackage{amsmath,amsfonts,amssymb,epsfig}
\usepackage{amsthm,amscd}
\usepackage{mathrsfs} 
\usepackage{hyperref}
\usepackage{cancel}
\usepackage{wrapfig}
\usepackage[mathcal]{eucal}
\usepackage{setspace}
\usepackage{accents}
\usepackage{color}
\usepackage{overpic}
\makeatletter
 \def\@textbottom{\vskip \z@ \@plus 11pt}
 \let\@texttop\relax
\makeatother

\makeatletter
\renewcommand\l@subsection{\@tocline{2}{0pt}{2pc}{5pc}{}}
\makeatother

\newcounter{mtheorem} % numerical counter 
\newcounter{mcorollary} % numerical counter 

\newtheorem{theorem}{Theorem}
\newtheorem*{theorem*}{Theorem}
\newtheorem*{observation*}{Observation}
\newtheorem{observation}[theorem]{Observation}
\newtheorem{corollary}[theorem]{Corollary}

\newtheorem{proposition}[theorem]{Proposition}
\newtheorem{lemma}[theorem]{Lemma}
\newtheorem*{klemma*}{Key Lemma}

\newtheorem{question}[theorem]{Question}

\newtheoremstyle{ourremark}% name
  {3pt}%      Space above
  {3pt}%      Space below
  {}%         Body font
  {}%         Indent amount (empty = no indent, \parindent = para indent)
  {\bfseries}% Thm head font
  {}%        Punctuation after thm head
  {.5em}%     Space after thm head: " " = normal interword space;
  {}%         Thm head spec (can be left empty, meaning `normal')
\theoremstyle{ourremark}
\newtheorem{remark}[theorem]{Remark}
\newtheorem{definition}[theorem]{Definition}
\newtheorem{example}[theorem]{Example}

\newlength{\dhatheight} % doublehat

\newcommand{\no}{\noindent}

\newcommand{\bex}{\begin{example}\em}
\newcommand{\eex}{\end{example}}

\def\R{\mathbb{R}}
\def\Z{\mathbb{Z}}

\graphicspath{{./pict/},{./../pict/}}

\numberwithin{equation}{section} 
\numberwithin{theorem}{section}
\title[Borsuk conjecture and homotopy domination]{On the Borsuk conjecture concerning\\
homotopy domination.} 

\date{\today}

\author{R. Komendarczyk}\thanks{The first author acknowledges the support of  DARPA YFA N66001-11-1-4132 and NSF DMS 1043009.}

\author{S. Kwasik}\thanks{The second author acknowledges the support of the Simons Foundation Grant No. 281810}

\address{Department of Mathematics,
Tulane University,
New Orleans, United States } 
\email{rako@tulane.edu, kwasik@tulane.edu}
\author{W. Rosicki}%\thanks{The third author acknowledges \ldots.}
\address{Faculty of Mathematics,
University of Gdansk, 
Gdansk, Poland} 
\email{wrosicki@mat.ug.edu.pl}
 
\begin{document}
%\onehalfspacing
\setcounter{section}{0}

\begin{abstract}
 In the seminal monograph {\em Theory of retracts}, Borsuk raised the following question: suppose  two compact ANR spaces are $h$--equal, i.e. mutually homotopy dominate each other, are they homotopy equivalent? The current paper approaches this question in two ways. On one end, we provide conditions on the fundamental group which guarantee a positive answer to the Borsuk question.  On the other end, we construct various examples of compact $h$--equal, not homotopy equivalent continua, with distinct properties. The first class of these examples has trivial all known algebraic invariants (such as homology, homotopy groups etc.) The second class is given by  $n$--connected  continua, for any $n$, which are infinite $CW$--complexes, and hence ANR spaces, on a complement of a point.  
\end{abstract}

\maketitle

\section{Introduction}\label{sec:intro}
Given two topological spaces $X$ and $Y$, $X$ is {\em homotopy dominated} by $Y$; denoted by $X\leq_h Y$, if and only if there exist maps $f:X\longrightarrow Y$ and $g:Y\longrightarrow X$,
such that $g\circ f\simeq \text{id}_X$. If $X\leq_h Y$ and $Y\leq_h X$, the spaces $X$ and $Y$ are called $h$--{\em equal}, the latter denoted by $X=_h Y$. In particular if $X$ is homotopy equivalent to $Y$, i.e. $X\simeq Y$, then they are $h$--equal. In the homotopy theory of Borsuk's ANR spaces, c.f. \cite{Borsuk67}, two basic problems are raised. Paraphrasing  Borsuk \cite{Borsuk67}, the first one can be stated as follows:
\begin{quote}
{\em  1) Is every compact ANR space homotopy equivalent to a finite CW-complex?}
\end{quote}
and the second one:
\begin{quote}
{\em 2) Are two $h$--equal compact ANR spaces homotopy equivalent? In other words, given compact ANR spaces $X$ and $Y$, does $X=_h Y$ imply $X\simeq Y$?}
\end{quote}
\no Both questions become less challenging if the compactness condition is relaxed, since the answer to the first question is positive \cite{Milnor59}, and negative for the second one, \cite{Stewart58}.  {\em Problem 1} (with the compactness assumption) became known as the {\em Borsuk conjecture} and attracted a considerable interest (c.f. \cite{Milnor59, Kirby-Siebenmann77, Quinn79, Bryant-Ferry-Mio-Weinberger96}) which culminated in the positive solution by West in \cite{West77}. In contrast, for the second question surprisingly little progress has been made over the years. 
 One of the goals of the current paper is to renew interest in  {\em Problem 2}. 

The paper consists of essentially two parts. In the first part, which is mostly of expository nature, we make some comments on the role of the fundamental group in {\em Problem 2}. By analogy to Hopfian groups, we define a notion of a {\em Hopfian pair} for $h$--equal spaces and make the following observation.
%%%%%%%%%%%%%%%%%%%%%%%%%%%%%%%%%%%%%%%%%%%%%%%%%%%%%%%%%%%%% 
\begin{observation}\label{thm:hopfian}
The pair of ANR spaces $X$, $Y$ is a Hopfian pair, if and only if, $X$ and $Y$ are homotopy equivalent.
\end{observation}
This observation is in essence a reformulation of the classical Whitehead theorem, but it helps to put  the Borsuk problem ({\em Problem 2}) in a proper perspective. In particular, it yields the following consequence
\begin{theorem}\label{thm:noetherian}
Suppose $X$, and $Y$ are $h$--equal ANR spaces, such that $\pi_1(X)$ or $\pi_1(Y)$ is Hopfian, where $X$ is compact, or more generally has finitely generated homology groups $H_k(X)$ for all $k$. If one of the group rings $\Lambda$ is a Noetherian ring, then $X$ and $Y$ is a Hopfian pair, and hence $X$ and $Y$ are homotopy equivalent.
\end{theorem}
The following corollary is well known to the experts \cite{Kwasik84, Hausmann:1987, Kolodziejczyk05} in this research area:
%%%%%%%%%%%%%%%%%%%%%%%%%%%%%%%%%%%%%%%%%%%%%%%%%%%%%%%%%%%%% 
\begin{corollary}\label{cor:polycyclic}
 Suppose $X$ and $Y$ are compact, $h$--equal ANR spaces with the polycyclic-by-finite fundamental groups, then $X$ and $Y$ is a Hopfian pair.
\end{corollary}
In Section \ref{sec:pi_1-role}, we also make several related observations in the context of Hopfian pairs, Poincar\'e complexes and $H$--spaces.

The second part, the main part of the paper,  is where we construct  $2$--dimensional continua which are $h$--equal but not homotopy equivalent, see Theorem \ref{thm:WH-WBH}, these constructions are inspired by \cite{Karimov-Repovs-Rosicki-Zastrow05} and \cite{Stewart58}. A basic building block of these examples is a well known ``topological broom'' pictured on Figure \ref{fig:brooms}.
An interesting feature of these constructions is that these spaces have trivial all basic known algebraic invariants, such as singular or \v{C}ech homology groups, homotopy groups etc. Consequently, to prove that the spaces are not homotopy equivalent  requires a more direct, approach via techniques of set theoretic topology. 
Further, in Theorem \ref{thm:WH-WBH},  we provide examples of pairs $\mathcal{S}_0$, $\mathcal{S}_1$ of $2n$--dimensional continua (for $n\geq 2$), modeled on the {\em Hawaiian earrings}, c.f \cite{Eda-Kawamura00},  and  satisfying:    
\begin{itemize}
\item[(a)] $\mathcal{S}_0$, $\mathcal{S}_1$ are {\em singular ANR spaces}, i.e. for specific points $s_0\in \mathcal{S}_0$, and $s_1\in \mathcal{S}_1$, complements $\mathcal{S}^\circ_0=\mathcal{S}_0-\{s_0\}$, $\mathcal{S}^\circ_1=\mathcal{S}_1-\{s_1\}$ are countable disjoint\footnote{A simple modification  of this construction (see Section \ref{sec:inf-wedges}) yields a path connected complement of analogs of $\mathcal{S}_0$ and $\mathcal{S}_1$.} sums of connected ANR spaces.
\item[(b)] $\mathcal{S}_0$, $\mathcal{S}_1$ are $(n-1)$--connected, and each connected component of $\mathcal{S}^\circ_0$ and $\mathcal{S}^\circ_1$ is locally contractible.
\item[(c)] $\mathcal{S}_0=_h\mathcal{S}_1$ but $\mathcal{S}_0\not\simeq\mathcal{S}_1$.
\end{itemize}
The point of this construction is to obtain examples of {\em compact} spaces 
which are as close as possible to ANR spaces.  The construction is a generalization of the earlier result in \cite{Stewart58} and relies on the fairly recent work in \cite{Eda-Kawamura00}.

\section*{Acknowledgements} \no We wish thank the anonymous referee for remarks and corrections which improved quality of the paper.

%%%%%%%%%%%%%%%%%%%%%%%%%%
\section{On a role of the fundamental group in Borsuk's problem.}\label{sec:pi_1-role}

\subsection{Hopfian pairs} We recall that a finitely presented group $G$ is called {\em Hopfian}, if every epimorphism $h:G\longrightarrow G$ is an isomorphism.
Analogously, given a ring $R$, a finitely generated $R$--module $M$ is {\em Hopfian}, if any module epimorphism $h:M\longrightarrow M$ is an 
isomorphism. Let $X$ and $Y$ be a pair of $h$--equal spaces then, from definition, there are maps 
%%%%%%%%%%%%%%%%%%%%%%%%%%
\begin{equation}\label{eq:XY-dominated}
\begin{split}
 & f:X\longrightarrow Y,\quad i:Y\longrightarrow X,\quad f\circ i\simeq \text{id}_Y,\\
 & g:Y\longrightarrow X, \quad j:X\longrightarrow Y, \quad g\circ j\simeq \text{id}_X.
\end{split} 
\end{equation}
In particular it implies that induced homomorphisms $f_\ast$, $g_\ast$ on the fundamental group and homology groups, are epimorphisms.
\begin{definition}\label{def:hopfian}
A pair of spaces $X$, $Y$ is called {\em Hopfian pair}, if and only if $X=_h Y$ and one of the epimorphisms $(g\circ f)_*$ or $(f\circ g)_*$ induced on the fundamental groups and homology modules from maps in \eqref{eq:XY-dominated} is an isomorphism.
\end{definition}
\no Note that, in the above definition, if one of the epimorphisms is an isomorphism the second one is an isomorphism as well. For convenience, let us restate Observation \ref{thm:hopfian}:
\begin{observation*}
The pair of ANR spaces: $X$, $Y$ is  Hopfian pair, if and only if, $X$ and $Y$ are homotopy equivalent.
\end{observation*}
\begin{proof}
 From the above definition, both compositions
 \[
 \begin{split}
  & g_* f_*=(g\circ f)_*:\pi_1(X)\longrightarrow \pi_1(X),\\
  & f_* g_*=(f\circ g)_*:\pi_1(Y)\longrightarrow \pi_1(Y),
 \end{split}
\]
are isomorphisms. Thus,  $f_*$ and $g_*$ are monomorphisms and consequently they have to be isomorphisms as well. The same reasoning applies to the module homomorphisms $f_*:H_k(X;\Lambda)\longrightarrow H_k(Y;\Lambda)$, $g_*:H_k(Y;\Lambda)\longrightarrow H_k(X;\Lambda)$, $\Lambda=\Z[\pi]$. As a consequence, maps $f$ and $g$ induce isomorhisms on $\pi_1$, and all homology with local coefficients, and the Whitehead Theorem implies that $f$ and $g$ are homotopy equivalences. Let $f:X\longrightarrow Y$ be a homotopy equivalence with the inverse $g:Y\longrightarrow X$. Then, $g\circ f\simeq \text{id}_X$ and $f\circ g\simeq \text{id}_Y$, then obviously the pair $X$, $Y$ is a Hopfian pair.  
\end{proof}
%%%%%%%%%%%%%%%%%%%%%%%%%%%%%%%%%%%%%%%%%%%%%%
\begin{remark}\label{rem:only-1}
Note that in the above observation, it suffices to only have one ANR space; $X$ or $Y$, then the result of Milnor \cite{Milnor59}, implies that the other space (homotopy dominated by the former) is also an ANR, up to homotopy. 
\end{remark}
%%%%%%%%%%%%%%%%%%%%%%%%%%%%%%%%%%%%%%%%%%%%%%
\begin{proof}[Proof of Theorem \ref{thm:noetherian} and Corollary \ref{cor:polycyclic}]
 Since the group rings are Noetherian rings, and modules $H_\ast(X;$ $\Z[\pi_1(X)])$ and $H_\ast(Y;\Z[\pi_1(Y)])$ are finitely generated, they in turn are Hopfian, c.f. \cite{Lam:1999}. Given that $X$ and $Y$ are $h$--equal they must form a Hopfian pair, implying Theorem \ref{thm:noetherian}.  Since polycyclic-by-finite groups are Hopfian and their group rings Noetherian,  Corollary \ref{cor:polycyclic}
 is a special case of Theorem \ref{thm:noetherian}, .
\end{proof}
The class of Hopfian groups is considerably larger than the polycyclic-by-finite groups. In particular, the following question is a weaker form of {\em Problem 2}.
%%%%%%%%%%%%%%%%%%%%%%%%%%%%%%%%%%%%%%%
\begin{question}\label{q:hopfian-pair}
Let $X$, $Y$ be finite CW--complexes (compact ANR's) such that  $X=_h Y$, suppose further $\pi_1(X)$ (and hence $\pi_1(Y)$) is Hopfian. Is $X$, $Y$ a Hopfian pair? 
\end{question}
\no The following example, guided by the results of \cite{Berridge-Dunwoody79, Harlander-Jensen06}, illustrates a delicate nature of the above question. Indeed, if $G=\pi_1(X)\cong \pi_1(Y)$ and $G$ is Hopfian, it is not necessarily the case that $X=_h Y$, even for $2$--dimensional $CW$--complexes $X$ and $Y$.
%%%%%%%%%%% hopfian example
\begin{example}
  Let $G=\langle x,y\,|\,x^2=y^3\rangle$ be the standard presentation for the fundamental group of the trefoil knot, and 
  let
 \[
  G_i=\langle x,y,\bar{x},\bar{y}\,|\,x^2=y^3,\, \bar{x}^2=\bar{y}^3,\, x^{2i+1}=\bar{x}^{2i+1},\, y^{3i+1}=\bar{y}^{3i+1}\rangle,\qquad i\in \mathbb{N},
  \]
be different presentations of $G$ (c.f. \cite{Berridge-Dunwoody79, Harlander-Jensen06}). For infinitely many $i$, there are $2$--dimensional $CW$--complexes $K_i$ of distinct homotopy type, with $\pi_1(K_i)\cong G_i\cong G$, c.f. \cite{Harlander-Jensen06}. Note that the commutator subgroup $[G,G]$ of $G$ is isomorphic to $F_2$, i.e. free group on two generators, and $G\bigl/[G,G]\cong H_1(G;\Z)\cong \Z$.

 Also, $G$ is Hopfian, since both $[G,G]\cong F_2$ and $G\bigl/ [G,G]\cong \Z$ are Hopfian and it is well known that $G$ is  {\em not} polycyclic--by--finite. We claim that there are infinitely many pairs $i$ and $j$, $i\neq j$, such that $K_i\neq_h K_j$. 

 Recall that it is shown in \cite{Berridge-Dunwoody79, Harlander-Jensen06} that there are infinitely many pairs $i$ and $j$, $i\neq j$, such that $K_i$ is not homotopy equivalent to $K_j$, because $H_2(\widetilde{K}_i;\Z)$ and $H_2(\widetilde{K}_j;\Z)$ are not isomorphic as $\Z[G]$--modules. More precisely, for some prime number $p$, there are infinitely many distinct $i$ and $j$ such that $\Z_p\otimes_{\Z} H_2(\widetilde{K}_i;\Z)$ has just one generator and  $\Z_p\otimes_{\Z} H_2(\widetilde{K}_j;\Z)$ has at least two generators. 
Suppose $K_i=_h K_j$, for $i\neq j$, where the above holds, then one obtains an epimorphism  $f_\ast:H_2(K_i;\Z[G])\longrightarrow H_2(K_j;\Z[G])$ (see Definition \ref{def:hopfian}), and therefore an obvious epimorphism
\[
 \text{id}\otimes f_\ast:\Z_p\otimes_{\Z} H_\ast(K_i;\Z[G])\longrightarrow \Z_p\otimes_{\Z} H_\ast(K_j;\Z[G]),
\]
this however contradicts that  $H_\ast(K_i;\Z[G])$ has just one generator versus $H_\ast(K_j;\Z[G])$ having two generators. Thus by contradiction, we conclude that $K_i\neq_h K_j$ and hence the pair $K_i$ and $K_j$ is not a Hopfian pair. 

 The above considerations lead one to a surprising outcome when one considers spaces $K_i\vee S^2$ and $K_j\vee S^2$ in place of $K_i$ and $K_j$. By work in 
\cite{Berridge-Dunwoody79, Harlander-Jensen06} we know that 
\[
 K_i\vee S^2\simeq K_j\vee S^2,\qquad\text{thus}\qquad K_i\vee S^2=_h K_j\vee S^2.
\]
Note that $G\cong \pi_1(K_i\vee S^2)\cong \pi_1(K_j\vee S^2)$, and the pair $K_i\vee S^2$, $K_j\vee S^2$ is  Hopfian. Indeed the modules $H_2(K_i\vee S^2;\Z[G])$, $H_2(K_j\vee S^2;\Z[G])$ are Hopfian as both are isomorphic to the free $\Z[G]$--module $\Z[G]\oplus\Z[G]$ (c.f. \cite{Berridge-Dunwoody79}). This shows a difficulty in dealing with modules  $H_\ast(X;\Z[G])$ and  $H_\ast(Y;\Z[G])$, in the context of Question \ref{q:hopfian-pair}, even if the group $\pi_1(X)\cong\pi_1(Y)$ is a ``nice'' group.
\end{example}
%
%
%%%%%%%%%%%%%%%%%%%%%%%%%%%%%%%%%%%%%%%%%
\subsection{Poincar\'e complexes}  Now, let $M^n$ be a closed $n$--dimensional manifold and $Y$ any space (see Remark \ref{rem:only-1}), such that $M^n=_h Y$, then $M^n$, $Y$ is a Hopfian pair \cite{Berstein-Ganea59, Kwasik84}. More generally, let $X$ be a finite Poincar\'e complex of formal dimension $n$, c.f. \cite{Wall65}. To be specific, $X$ has a homotopy type of a finite CW--complex and there exists a class $[X]\in H_n(X;\Z)$, such that for all $r$ the cap product with $[X]$ induces an isomorphism
\[
 [X]\cap\cdot :H^r(X;\Lambda)\longrightarrow H_{n-r} (X;\Lambda),\qquad \Lambda=\Z[\pi_1(X)].
\]
If $Y$ is any space, such that $X=_h Y$ then $X$, $Y$ is a Hopfian pair, \cite{Kwasik84}. 
%%%%%%%%%%%%%%%%%%%%%%%%%%%%%%%%%%%%%%%%%
\begin{theorem}
 Suppose $X$ is a homology manifold of formal dimension $n$, i.e. $X$ is a finite dimensional ANR space such that
\[
H_\ast(X,X-\{\text{pt}\})\cong H_\ast(\R^n,\R^n-\{\text{pt}\})=\begin{cases}
\Z,\quad \ast=n,\\
0,\quad \ast\neq 0.
\end{cases}
\]
 Then $X$ is a finite Poincar\'e complex of formal dimension $n$.
\end{theorem}
%%%%%%%%%%%%%%%%%%%%%%%%%%%%%%%%%%%%%%%%%
 The above theorem is stated without a proof in \cite[p. 5099]{Johnston-Ranicki-00}. It is a well known fact that $X$ satisifies the Poincar\'e duality with integer coefficients, \cite{Borel57}. The only argument we are aware of, that shows $X$ is a Poincar\'e complex, is based on the existence of a spectral sequence for the indentity map $\text{id}_X: X\longrightarrow X$ in sheaf homology giving a very general version of Poincar\'e duality in Theorem 9.2 of \cite{Bredon-book97}. It should be noted that if $X$ is {\em polyhedral homology manifold} then a much simpler argument shows that $X$ is a Poincar\'e complex (see Theorem 2.1 in \cite{Wall65}).
%%%%%%%%%%%%%%%%%%%%%%%%%%%%%%%%%%%%%%%%%
\begin{corollary}
Let $X$ be a homology manifold of formal dimension $n$ and $Y$ any space with $X=_h Y$, then $X$, $Y$ is a Hopfian pair.
\end{corollary}
Recall, that the well known conjecture asserts that finite dimensional homogeneous ANRs are homology manifolds, \cite{Bryant-Ferry-Mio-Weinberger96}.

Following, \cite{Bredon70}, recall that  $X$ is {\em locally isotopic} if for each path $\lambda:[0,1]\longrightarrow X$, there is a neighborhood $N$ of $\lambda(0)$ in $X$ and a map $H:I\times N\longrightarrow X$, such that $H(t,\lambda(0))=\lambda(t)$ and such that each $H(t,\,\cdot\,)$ is a homeomorphism of $N$ onto a neighborhood of $\lambda(t)$. Clearly, manifolds are locally isotopic. Suppose $X$ is a compact finite dimensional ANR space which is locally isotopic. By Theorem 4.6 of \cite{Bredon70}, $X$ is a homology manifold of some formal dimension  $n$. Thereore, we obtain
%%%%%%%%%%%%%%%%%%%%%%%%%%%%%%%%%%%%%%%%%
\begin{corollary}
Let $X$ be a compact finite dimensional ANR space which is locally isotopic, and let $Y$ any space such that $X=_h Y$. Then $X$, $Y$ is a Hopfian pair.
\end{corollary}

\no In the case $X$ admits an $H$--space structure, $\pi_1(X)$ is abelian, in particular polycyclic-by-finite, thus if  $H_\ast(X)$ to be finitely generated in each degree (where we allow the degree to go to infinity), we obtain 
%%%%%%%%%%%%%%%%%%%%%%%%%%%%%%%%%%%%%%%%%
\begin{proposition}
 Let $X$ be an $H$--space, such that $H_k(X)$ is finitely generated for each $k$, and $Y$ any space such that $X=_h Y$. Then $X$, $Y$ is a Hopfian pair.
\end{proposition}
\no Clearly,  if $X$ is a compact $H$--space the above homological condition holds. Curiously enough, compact $H$--spaces are also Poincar\'e complexes, as can be deduced from the work in \cite{Bauer-et.al04}.

%%%%%%%%%%%%%%%%%%%%%%%%%%%%%%%%%%%%%%%%%%%%%%%%%%%%%%%%%%%%%
\section{About \texorpdfstring{$h$}{h}--equal but not homotopy equivalent spaces}\label{sec:continua}
%%%%%%%%%%%%%%%%%%%%%%%%%%%%%%%%%%%
Looking for a counterexample to {\em Problem 2}, one may consider the following problem in the combinatorial group theory;  suppose $G$ and $H$, $G\not\cong H$ are two finitely presented groups and retracts of each other, which would make such pair of groups ``strongly'' non--Hopfian.
If both $G$ and $H$ are finite dimensional, i.e. $K(G;1)$ and $K(H;1)$ are chosen to be finite $CW$--complexes, then the functoriality of the construction of $K(\pi;1)$--spaces would imply the existence of a counterexample to {\em Problem 2}, namely
\[
 K(G;1)=_h K(H;1),\quad \text{and}\quad K(G;1)\not\simeq K(H;1).
\]
\no Consequently, the following algebraic question is of crucial importance and of an independent interest.
%%%%%%%%%%%%%%%%%%%%%%%%%%%%%%%%%%%%%%%
\begin{question}
Find two finitely presented groups  $G$ and $H$, such that $G\not\cong H$ which are retracts of each other.
\end{question}

If one considers a more general class of spaces,  then the answer to {\em Problem 2} is negative, as first observed by Stewart in \cite{Stewart58}, who provided examples of noncompact ANR spaces. The remainder of this paper is devoted to a construction of compact examples with particular properties as described in the introduction, Section \ref{sec:intro}.
%
%
%%%%%%%%%%%%%%%%%%%%%%%%%%%%%%%%%%%%%%%
\subsection{Infinite wedges of ``hairy disks''}
Our example is inspired by constructions of both \cite{Karimov-Repovs-Rosicki-Zastrow05} and \cite{Stewart58}, and based on the ``hairy disk'' depicted in Figure \ref{fig:hairy-disk}.
First, consider a double broom $\mathcal{B}$ as shown on Figure \ref{fig:brooms}.  $\mathcal{B}$ is a well known  space which is not contractible but has all trivial known algebraic invariants, such as homology and homotopy groups etc. \cite[p. 295]{Hilton-Wylie-book}. 
Denote the {\em center point} of the broom $\mathcal{B}$ by $v$ and the left and right sequence of broom's endpoints converging to $v$ by $\{a_n\}$ and $\{b_n\}$ respectively.  Generally, $J_x$, provided it is uniquely determined, will refer to a segment of $\mathcal{B}$ containing $x\in\mathcal{B}$. An exception to this are the following cases:  for $x=v, a_0, b_0$, we set
\begin{equation}\label{eq:J-arms}
 J_v=[v,a_0]\cup [v,b_0],\quad  J_{a_0}=[v,a_0],\quad  J_{b_0}=[v,b_0].
\end{equation}
In particular,
\begin{equation}\label{eq:J_a-J_b-arms}
	 J_{a_n}=[a_n, a_0],\quad J_{b_n}=[b_n, b_0].
\end{equation}
\no Naturally, we may view $\mathcal{B}$ as a wedge product of two pieces $A$ and $B$, containing sequences $\{a_n\}$ and $\{b_n\}$, i.e.
\begin{equation}\label{eq:B-union}
 \mathcal{B}=A\vee B,\qquad A=J_{a_0}\cup\bigcup_{n} J_{a_n},\quad  B=J_{b_0}\cup\bigcup_{n} J_{b_n}.
\end{equation}
\no Further, we order points in $\mathcal{B}$ along segments $J_x$; simply assuming the order is ``increasing'' from the bottom to top, for instance any $z\in J_{a_n}$ satisfies $a_n\leq z\leq a_0$. In particular, if  $x, y\in J_{w}$, and $x\leq y$ with respect to this order, then $[x,y]$ will denote a portion of the segment $J_{w}$ containing all $z$ such that $x\leq z\leq y$.
%
%%%%%%%%%%%%%%%%%%%%%%%%%%%%%%%%%%%%%
\begin{figure}[!ht] 
  \centering  
 \includegraphics[width=0.6\textwidth]{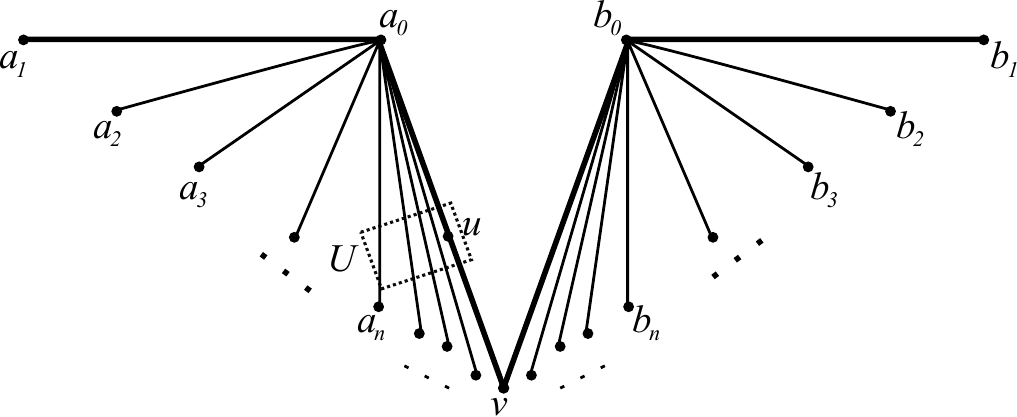}
  \caption{Topological broom denoted by $\mathcal{B}$, with the center point $v$, and a location of the point $u$, and its neighborhood $U$, considered in the proof of Lemma \ref{lem:g-f-hf}.} \label{fig:brooms} 
\end{figure} 

%%%%%%%%%%%%%%%%%%%%%%%%%%%%%%%%%%%%%%%%%
\begin{definition}\label{def:h.f.}
Given a space $X$, we say $x\in X$ is {\em homotopically fixed in $X$}, if $x$ is fixed under any homotopy $f_t$, where $f_0=\text{id}_X$. The set of homotopically fixed points in $X$ is denoted by
\begin{equation}\label{eq:hf(X)}
 hf(X)=\{x\in X\ |\ x\ \text{is  homotopically fixed in $X$}\}.
\end{equation}
\end{definition}
\no The set $hf(X)$ is a closed subset of $X$, in particular we have the following fact about $\mathcal{B}$:

%%%%%%%%%%%%%%%%%%%%%%%%%%%%%%%%%%%%%
\begin{lemma}\label{lem:v-fixed-B}
 Suppose $f:\mathcal{B}\longrightarrow \mathcal{B}$ fixes $v$, i.e. $f(v)=v$, and $v$ is a limit point for both sets: $Z \cap A$ and $Z \cap B$, $Z=f(\mathcal{B})$. Then, any homotopy $f_t:\mathcal{B}\longrightarrow \mathcal{B}$, $f_0=f$ keeps $v$ fixed, i.e. $f_t(v)=v$.
\end{lemma}
%%%%%%%%%%%%%%%%%%%%%%%%%%%%%%%%%%%%
\begin{proof}[Sketch of Proof]
 Choose sequences $\{u_n\}$, $u_n\in J_{a_n}$ in $Z\cap A$  and $\{w_n\}$, $w_n\in J_{b_n}$ in $Z\cap B$ respectively, such that
 \[
  u_n\longrightarrow v,\qquad\text{and}\qquad w_n\longrightarrow v\qquad \text{in}\quad \mathcal{B}.
 \]
\no Let $v_t=f_t(v)$, by continuity, for each $t$:
\[
  f_t(u_n)\longrightarrow v_t,\qquad\text{and}\qquad f_t(w_n)\longrightarrow v_t.
 \]
	\no Note that a point $u_n$ can only move up along the arm $J_{u_n}\subset A$ of $\mathcal{B}$ and $w_n$ move up along $J_{w_n}\subset B$ (for large enough $n$).  Thus $v_t=v_0$ for all $t$, because $A\cap B=\{v\}$.
\end{proof}

\no The above lemma is completely analogous to \cite[Lemma 2.3]{Karimov-Repovs-Rosicki-Zastrow05}, where a similar topological broom  is considered\footnote{We choose the broom $\mathcal{B}$, shown on Figure \ref{fig:brooms}, over the one constructed in \cite{Karimov-Repovs-Rosicki-Zastrow05} to simplify certain arguments of this section.}.

\begin{corollary}\label{lem:hf(B)}
 $hf(\mathcal{B})=\{v\}$.
\end{corollary}
\begin{proof}%[Sketch of Proof]
 We already know that $v\in hf(\mathcal{B})$ by Lemma \ref{lem:v-fixed-B}. It is easy to rule out other points in $\mathcal{B}$ as homotopically fixed, with an exception of possibly $a_0$ and $b_0$. Observe however, that $a_0$ and $b_0$ cannot be homotopically fixed as we may construct a homotopy which lets $a_0$ or $b_0$ to ``flow out'' along one of the arms of $\mathcal{B}$, e.g $J_{a_1}$ and  $J_{b_1}$ respectively.
\end{proof}
 Before introducing relevant spaces we make the following convenient definition of a {\em wedge product} $\curlyvee$ of spaces $X$ and $Y$ disjointly embedded in $\R^N$ (for some $N$): 
\begin{equation}\label{eq:new-wedge}
 X\curlyvee_{x,y} Y:= X\cup [x,y]\cup Y,\qquad x\in X,\ y\in Y,
\end{equation}
where $[x,y]$ is an arc in $\R^N$ connecting points $x$ and $y$, with its interior $(x,y)$ disjoint from $X$ and $Y$. 
%
%
%%%%%%%%%%%%%%%%%%%%%%%%%%%%%%%%%%%%%
\begin{figure}[!ht] 
  \centering
   \includegraphics[width=0.3\textwidth]{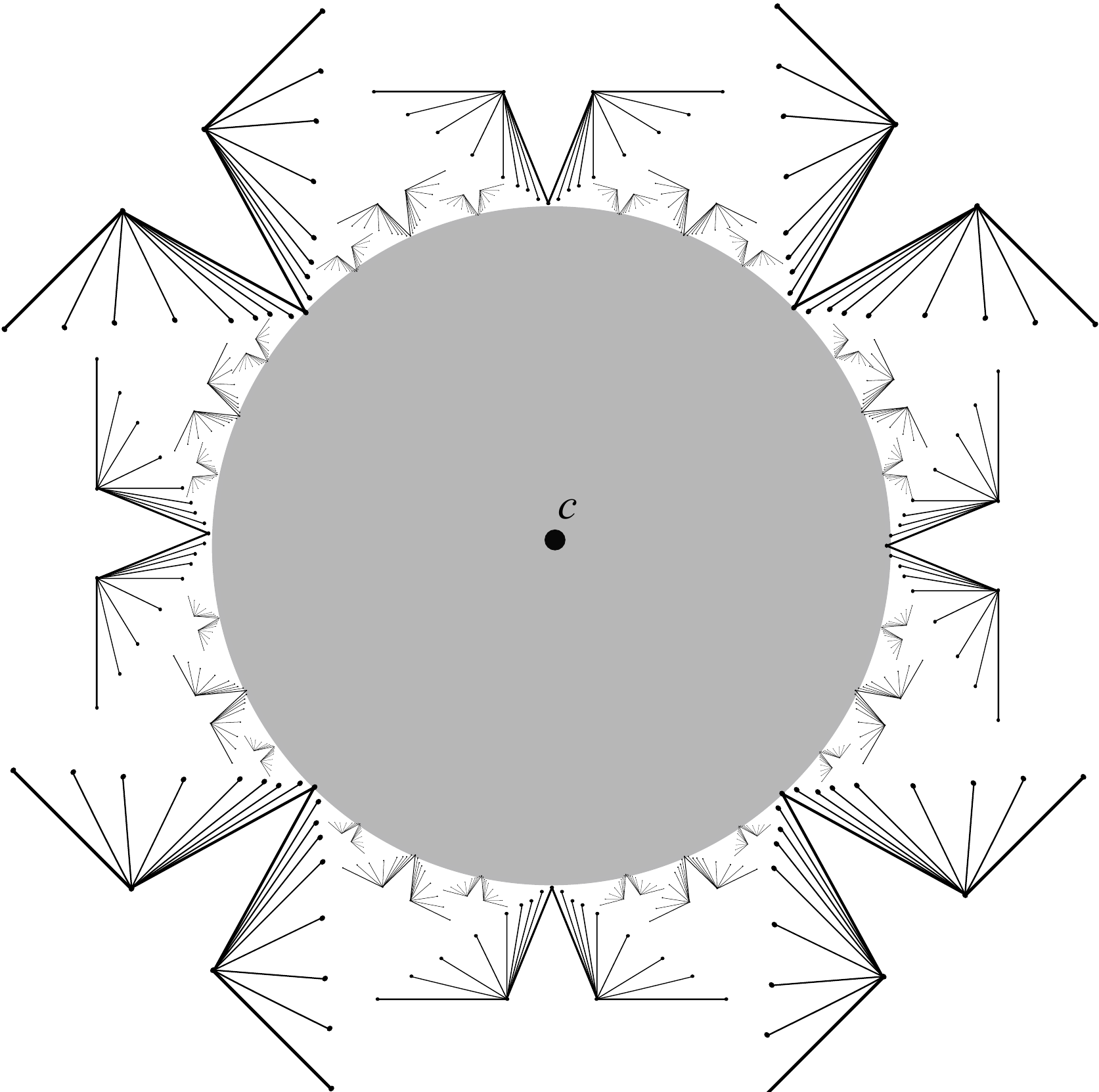} \qquad
  \caption{The ``hairy disk'' $\mathcal{H}$ from \cite[p. 286]{Karimov-Repovs-Rosicki-Zastrow05}, with countably many copies of $\mathcal{B}$ densely attached along the boundary of the unit disk in $\R^2$.} \label{fig:hairy-disk} 
\end{figure} 
%%%%%%%%%%%%%%%%%%%%%%%%%%%%%%%%%%%%%
The disk $\mathcal{H}$ is constructed by densely attaching brooms $\mathcal{B}$, Figure \ref{fig:brooms}, along the boundary of the unit disk $D^2$ in $\R^2$, c.f. \cite{Karimov-Repovs-Rosicki-Zastrow05}. More precisely, let 
\begin{equation}\label{eq:M-def}
M =\{m_i\}^{\infty}_{i=1}
\end{equation}
to be a countable dense subset of the boundary of $D^2$. Then $\mathcal{H}$ is obtained by attaching to each
point $m_i$ the broom $\mathcal{B}$ at the center vertex $v$, so that the copies of $\mathcal{B}$ do not intersect each
other and their diameters tend to zero as $i\to\infty$. Denote by $c$ the center of interior disk $D^2$ in $\mathcal{H}$. Following the ideas of \cite{Stewart58} we consider a countable wedge product of $\mathcal{H}$, with center points at $c(k)=(-\frac{1}{k},0,0)$ along the $x$--axis of $\R^3$. Each copy of $\mathcal{H}$ is denoted by $\mathcal{H}(k)$ and contained in the translated $yz$--plane to the point $c(k)$, together with the connecting segments $[c(k),c(k+1)]$ along the $x$--axis. In addition, each factor $\mathcal{H}(k)$ is scaled down to have the diameter $\frac{1}{k}$. This process yields a non--compact space we denote by $\mathcal{WH}^\circ$. The second space, denoted by $\mathcal{WBH}^\circ$ is obtained from $\mathcal{WH}^\circ$ by wedging a copy of $\mathcal{B}$ at the first factor. Using the notation in \eqref{eq:new-wedge}, we express these spaces as follows
%
%
%%%%%%%%%%%%%%%%%%%%%%%%%%%%%%%%%%%%%
\begin{figure}[!ht] 
  \centering
   \includegraphics[width=0.45\textwidth]{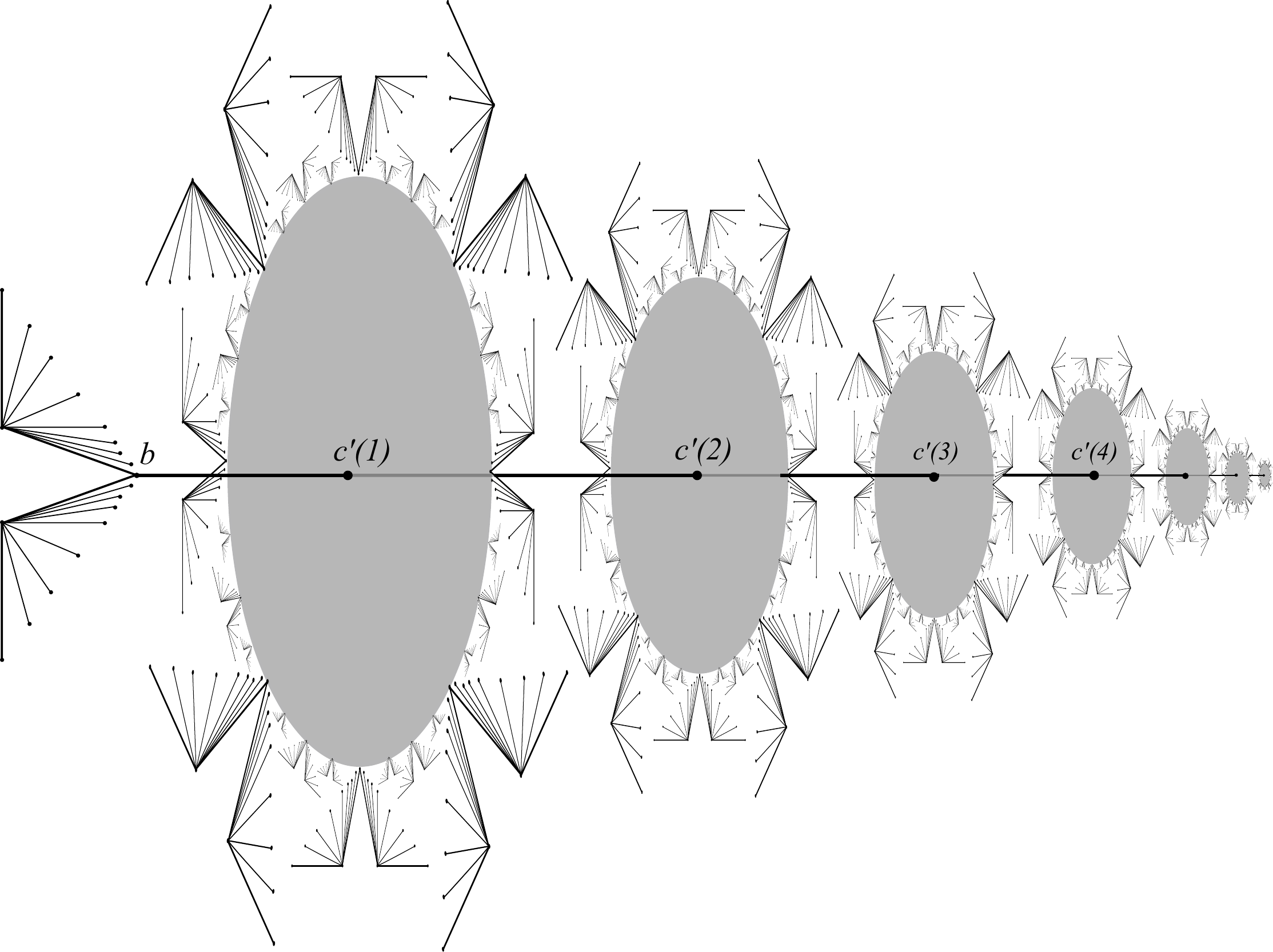}\quad \includegraphics[width=0.45\textwidth]{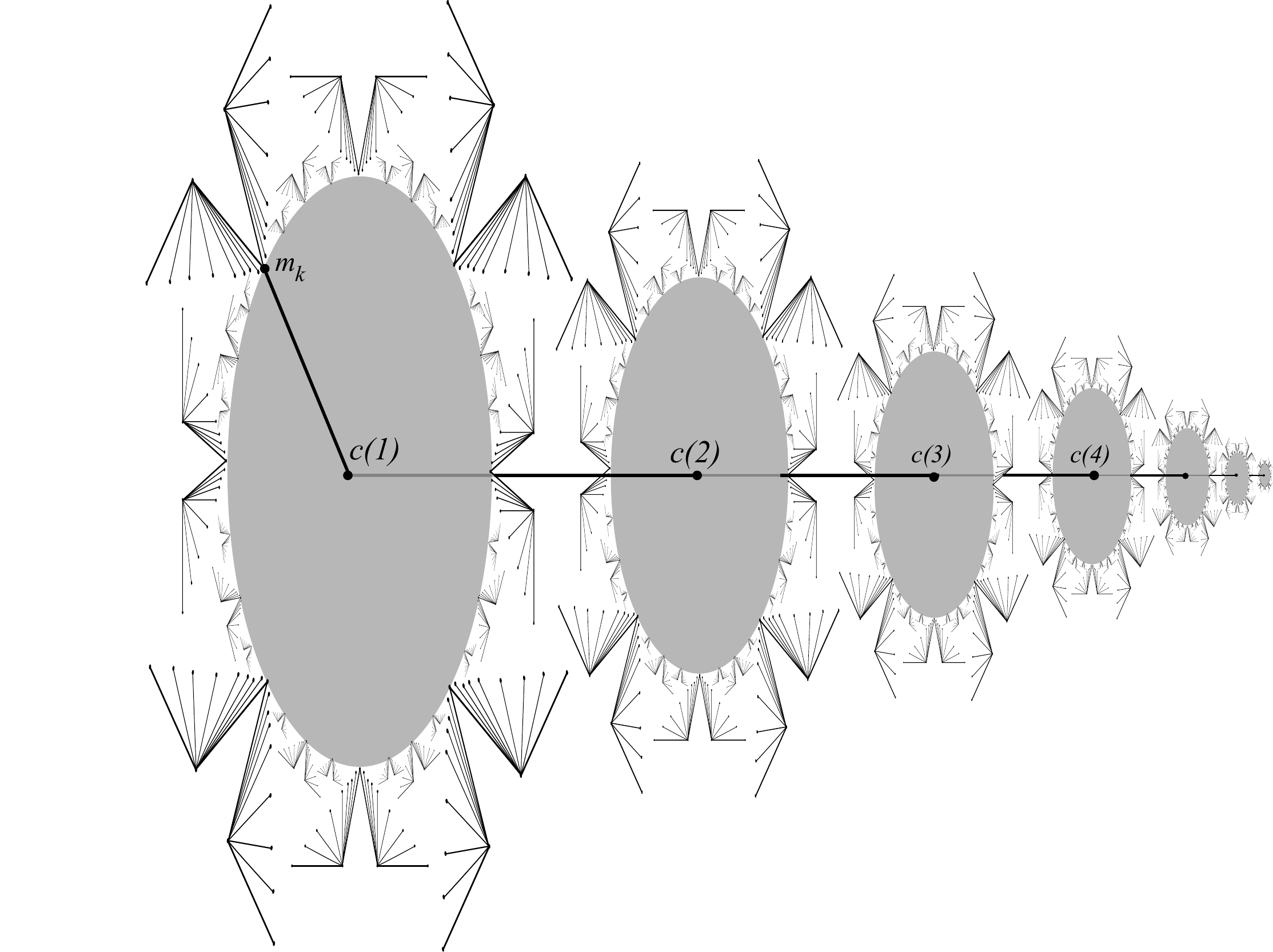}
  \caption{Infinite wedge products of hairy disks: $\mathcal{WBH}^\circ$(left) and $\mathcal{WH}^\circ$(right) embedded in $\R^3$. The segment connecting $c(1)$ and $m_k$, in the right picture, indicates a choice of embedding $\mathcal{WBH}^\circ\hookrightarrow \mathcal{WH}^\circ$.} \label{fig:WH} 
\end{figure} 
%%%%%%%%%%%%%%%%%%%%%%%%%%%%%%%%%%%%%%%%%%%%%%%%%%%%%%%%%%%%%%%%%%%%%%%%%%%%%%%%%%%%%%%
\begin{equation}\label{eq:WH-WBH-o}
\begin{split}
  \mathcal{WH}^\circ & =\mathcal{H}(1)\curlyvee_{c(1),c(2)}\mathcal{H}(2)\curlyvee_{c(2),c(3)}\cdots\curlyvee_{c(k-1),c(k)}\mathcal{H}(k)\curlyvee_{c(k),c(k+1)}\cdots,\\
  \mathcal{WBH}^\circ & =\mathcal{B}\curlyvee_{b,c'(1)} \mathcal{H}'(1)\curlyvee_{c'(1),c'(2)}\cdots\curlyvee_{c'(k-1),c'(k)}\mathcal{H}'(k)\curlyvee_{c'(k),c'(k+1)}\cdots,
\end{split}
\end{equation}
\no where $b$ denotes the center point $v$ of the $\mathcal{B}$ factor, and $\mathcal{H}$--factors of $\mathcal{WBH}^\circ$ are denoted by $\mathcal{H}'(k)$ for clarity.
The countable dense subset $M$ of $\mathcal{H}$ given in \eqref{eq:M-def} will be further denoted by $M(k)$ for each  $\mathcal{H}(k)$ in $\mathcal{WH}^\circ$, and by $M'(k)$ for each $\mathcal{H}'(k)$ in $\mathcal{WBH}^\circ$. Including the origin of $\R^3$ in both $\mathcal{WH}^\circ$ and $\mathcal{WBH}^\circ$ yields
\begin{equation}\label{eq:WH-WBH}
\mathcal{WH}=\mathcal{WH}^\circ\cup\{(0,0,0)\},\qquad \mathcal{WBH}=\mathcal{WBH}^\circ\cup\{(0,0,0)\}.
\end{equation}
\no Observe that $\mathcal{WH}$ and $\mathcal{WBH}$ are homeomorphic with one--point compactifications of  $\mathcal{WH}^\circ$ and $\mathcal{WBH}^\circ$. 
 We claim $\mathcal{WH}^\circ$ and $\mathcal{WBH}^\circ$ are $h$--equal, in particular there exist retractions\footnote{thus, $\mathcal{WH}^\circ$ and $\mathcal{WBH}^\circ$ are $r$--equal, c.f. \cite{Borsuk67}.} 
\begin{equation}\label{eq:retractions}
 r_{\mathcal{WH}}:\mathcal{WBH}^\circ\longrightarrow \mathcal{WH}^\circ,\qquad r_{\mathcal{WBH}}:\mathcal{WH}^\circ\longrightarrow \mathcal{WBH}^\circ.
\end{equation}
 Since $\mathcal{WH}$ is naturally a subset of $\mathcal{WBH}$, $r_{\mathcal{WH}}$ can be chosen as a quotient projection mapping the $\mathcal{B}$--factor, together with the segment $[b,c'(1)]$, to the point $c(1)\in \mathcal{H}(1)$ of $\mathcal{WH}^\circ$. 
 The retraction $r_{\mathcal{WBH}}$ of $\mathcal{WH}^\circ$ onto $\mathcal{WBH}^\circ$ can be defined once we choose an embedding 
 $\iota:\mathcal{WBH}^\circ \hookrightarrow\mathcal{WH}^\circ$. Once a point $m_k\in M(1)$ is selected, the embedding $\iota$ can be chosen to map the $\mathcal{B}$--factor of $\mathcal{WBH}$ to the factor $\mathcal{B}_{m_k}$ of $\mathcal{H}(1)$ and identifying the segment $[b,c'(1)]$ with the segment $[m_k,c(1)]$ in $\mathcal{H}(1)$, as shown on Figure \ref{fig:WH}(right). 
  Having identified $\mathcal{WBH}^\circ$ with a subset of $\mathcal{WH}^\circ$ we may define the retraction $r_{\mathcal{WBH}}$ next. It will be done in two stages; first we retract all the broom factors of $\mathcal{H}(1)$, except $\mathcal{B}_{m_k}$ to the boundary circle of the interior disk $D^2$ in $\mathcal{H}(1)$ via the following obvious map
  \[
	z:\mathcal{WH}^\circ\longrightarrow\mathcal{WH}^\circ,\qquad z(x)=\begin{cases}
	m_i & x\in \mathcal{B}_{m_i}\subset \mathcal{H}'(1),\quad i\neq k\\
	x, & \text{otherwise}.
	\end{cases}
  \]
  Continuity of $z$ is a direct consequence of the ``hairy disk'' construction. Indeed given a convergent sequence of points $\{h_n\}$ in the complement of $\mathcal{B}_{m_k}$, i.e. $\{h_n\}\subset \mathcal{H}(1)-\mathcal{B}_{m_k}$. If the limit of $\{h_n\}$ is in $\mathcal{B}_{m_k}$ then it has to be $m_k$, proving continuity of $z$. For the second stage, in the construction of $r_{\mathcal{WBH}}$, we define a map $y:z(\mathcal{WH}^\circ)\longrightarrow \mathcal{WBH}^\circ$ on the image of $z$,
  simply by collapsing the interior disk $D^2\subset \mathcal{H}(1)$ onto the segment $[m_k, c(1)]$ in $\mathcal{H}(1)$. The required retraction  $r_{\mathcal{WBH}}$ can be now defined as  $r_{\mathcal{WBH}}=y\circ z$. 
 \begin{theorem}\label{thm:WH-WBH}
 Both pairs: $\mathcal{WH}^\circ$, $\mathcal{WBH}^\circ$ and  $\mathcal{WH}$, $\mathcal{WBH}$ are $h$--equal but not homotopy equivalent.
\end{theorem}
\no The proof requires the following two lemmas.  

%%%%%%%%%%%%
\begin{lemma}\label{lem:hf(WH)-hf(WBH)}
 We have the following homeomorphisms 
\begin{equation}\label{eq:hf(H_0)-hf(H_1)}
\begin{split}
   hf(\mathcal{WH}^\circ) & \cong \bigsqcup^\infty_{i=1} hf(\mathcal{H}(k)),\quad
   hf(\mathcal{WBH}^\circ) \cong \{b\}\sqcup \bigsqcup^\infty_{i=1} hf(\mathcal{H}'(k)),\\
   hf(\mathcal{WH}) & \cong hf(\mathcal{WH}^\circ) \sqcup \{(0,0,0)\},\quad
   hf(\mathcal{WBH}) \cong hf(\mathcal{WBH}^\circ) \sqcup \{(0,0,0)\},
   \end{split}
\end{equation}
where each $hf(\mathcal{H}(k))$ or $hf(\mathcal{H}(k))$ is just a boundary of the interior  disk $D^2$ in each factor $\mathcal{H}(k)$ of $\mathcal{WH}^\circ$ ($\mathcal{WH}$) or $\mathcal{H}'(k)$ in $\mathcal{WBH}^\circ$ ($\mathcal{WBH}$), and therefore homeomorphic to $S^1$. The topology is the subspace topology induced from $\R^3$ via the embeddings constructed in \eqref{eq:WH-WBH-o}. 
\end{lemma}
\begin{corollary}\label{cor:hf(WH)-hf(WBH)}
	In particular, $hf(\mathcal{WH}^\circ)$ ($hf(\mathcal{WH})$) is not homeomorphic to $hf(\mathcal{WBH}^\circ)$ ($hf(\mathcal{WBH})$).
\end{corollary}
\begin{proof}[Proof of Lemma \ref{lem:hf(WH)-hf(WBH)}]
 Lemma \ref{lem:hf(B)} and the construction of the hairy disk $\mathcal{H}$ imply\footnote{It is easy to see that broom centers along $\mathcal{H}$--factors cannot be moved, by a homotopy, to the interior of the disk $D^2\subset\mathcal{H}$, c.f. \cite[$(v)$ on p. 288]{Karimov-Repovs-Rosicki-Zastrow05}.}
 \[
 M\subset hf(\mathcal{H}). 
 \]
  Since $M$ is dense in the boundary $S^1=\partial D^2\subset\mathcal{H}$, we obtain  
 \[
 S^1=\overline{M}\subset hf(\mathcal{H}).
 \]
  Since, none of the interior points in $D^2\subset \mathcal{H}$ is homotopically fixed, and by Lemma \ref{lem:hf(B)}, for each $\mathcal{B}_{m_k}$--factor of $\mathcal{H}$, $m_k$ is the only homotopically fixed point of $\mathcal{B}_{m_k}$, we conclude 
  \[
  hf(\mathcal{H})=S^1.
  \]
  It in turn implies equalities in \eqref{eq:hf(H_0)-hf(H_1)}, note that $\{(0,0,0)\}$ is homotopically fixed as a limit of  points in $hf(\mathcal{H}(k))$ fixed points from the $\mathcal{H}$--factors of 
$\mathcal{WH}$ or $\mathcal{WBH}$.
\end{proof}
\no Further, we obtain the following key lemma,
\begin{lemma}\label{lem:g-f-hf}
Let $f$ be the homotopy equivalence between $\mathcal{WH}^\circ$,  and $\mathcal{WBH}^\circ$, and $g$ its inverse. Then, 
\begin{equation}\label{eq:g-f-hf}
 f(hf(\mathcal{WH}^\circ))\subset hf(\mathcal{WBH}^\circ),\qquad g(hf(\mathcal{WBH}^\circ))\subset hf(\mathcal{WH}^\circ).
\end{equation}
The same inclusions holds for the compactifications: $\mathcal{WH}$ and $\mathcal{WBH}$.
\end{lemma}
\begin{proof}
We will prove the first inclusion in \eqref{eq:g-f-hf}, as the proof of the second is analogous. It suffices to prove for each $k$: 
\begin{equation}\label{eq:f(M(k))-in-hf}
 f(M(k))\subset hf(\mathcal{WBH}^\circ).
\end{equation}
 Then the claim follows from continuity of $f$, and the fact that the closure of $\bigcup_k M(k)$ in $\mathcal{WH}^\circ$ is equal to $hf(\mathcal{WH}^\circ)$ (see Lemma \ref{lem:hf(WH)-hf(WBH)}). (Note that for the second inclusion in \eqref{eq:g-f-hf}, the only difference is the point $b$ (the center of the first broom factor of 
$\mathcal{WBH}^\circ$) which needs to be added to the union $\bigcup_k M(k)$). 

To prove \eqref{eq:f(M(k))-in-hf}, consider a point $v$ in $M(k)$. By definition it has to be the center point of one of a broom factors in $\mathcal{H}(k)\subset \mathcal{WH}^\circ$, see \eqref{eq:WH-WBH-o}. We  further denote this factor by $\mathcal{B}$ (i.e. $v\in \mathcal{B}\subset \mathcal{H}(k)$). Let $u=f(v)$, and suppose by contradiction $u\not\in hf(\mathcal{WBH}^\circ)$, then 

either\ 1$^\circ$, $\mathcal{WBH}^\circ$ is locally path connected at $u$; 

or\ 2$^\circ$, $\mathcal{WBH}^\circ$ is not locally path connected at $u$. 
 
\no {\em Observation \eqref{eq:in-J}:} Since $v\in hf(\mathcal{WH})$, we must have $g\circ f(v)=v$ (as $g\circ f\simeq \text{id}_{\mathcal{WH}}$). Consider sequences $a_n\to v$, $b_n\to v$ of points in $\mathcal{B}$ (see Figure \ref{fig:brooms}). Denote by $\tilde{a}_n=g\circ f(a_n)$, $\tilde{b}_n=g\circ f(b_n)$, then we have $\tilde{a}_n\to v$ and $\tilde{b}_n\to v$. We claim that for large $n$: 
%%%%%%%%%%%%
\begin{equation}\label{eq:in-J}
 \tilde{a}_n\in J_{a_n},\qquad \tilde{b}_n\in J_{b_n}.
\end{equation}
\begin{proof}
 Indeed, denoting the homotopy $g\circ f\simeq \text{id}_{\mathcal{WH}}$ by $h_t=h(t,\,\cdot\,)$, $h:I\times \mathcal{WH}\longrightarrow \mathcal{WH}$ we observe that for every $n$: 
$\gamma_{a_n}(t)=h_t(a_n)$ defines a path in $\mathcal{WH}$ connecting $a_n=\gamma_{a_n}(1)$ and $\tilde{a}_n=\gamma_{a_n}(0)=g\circ f(a_n)$ (analogously for the sequence $\{b_n\}$). Since for the limit point $v=\lim a_n$, $\gamma_v$ is a constant path, for a small $\varepsilon$--ball $B_v(\varepsilon)$ around $v$, the inverse image $h^{-1}(B_v(\varepsilon))\subset I\times \mathcal{WH}$ contains $I\times \{v\}$ and therefore some small neighborhood $I\times B_v(\delta)$ is also in $h^{-1}(B_v(\varepsilon))$.
For large enough $n$, $a_n$'s are in $B_v(\delta)$ and hence the paths $\gamma_{a_n}$ have image in $B_v(\varepsilon)$. It follows that each  $\gamma_{a_n}$ is contained in the connected component $J_{a_n}\cap B_v(\varepsilon)\subset \mathcal{B}$ of $B_v(\varepsilon)$. Hence, for small positive $\varepsilon$ we obtain $J_{a_k}\cap J_{a_j}\cap B_v(\varepsilon)=\emptyset$ and the first part of \eqref{eq:in-J}. The second part follows analogously. 
\end{proof}
\no Now we consider Case 1$^\circ$  and Case 2$^\circ$.

{\em Case 1$^\circ$}: Suppose $\mathcal{WBH}^\circ$ is locally path connected at $u=f(v)$. Choose a small path connected ball $B_u(\tilde{\varepsilon})$ around $u$, such that $f(B_v(\delta))\subset B_u(\tilde{\varepsilon})$ then  $g(B_u(\tilde{\varepsilon}))\subset B_v(\varepsilon)$ with $\delta$ and $\varepsilon$ chosen as in the proof of Observation \eqref{eq:in-J} above. Since $g(B_u(\tilde{\varepsilon}))$ is connected, and all $\{\tilde{a}_n\}$ for large $n$ are contained in  $g(B_u(\tilde{\varepsilon}))$, $\{\tilde{a}_n\}$ would have to belong entirely to one of the arms $J_{a_k}\cap B_v(\varepsilon)$ of $\mathcal{B}$. But, this leads to a contradiction with Observation \eqref{eq:in-J}.

{\em Case 2$^\circ$:} Suppose $\mathcal{WBH}^\circ$ is not locally connected at $u=f(v)$. Since $u\not\in hf(\mathcal{WBH}^\circ)$, 
$u$ belongs to one of the broom factors of $\mathcal{WBH}^\circ$, we denote by $\mathcal{B}'$ (i.e. $\mathcal{B}'$ is either the $\mathcal{B}$--factor of $\mathcal{WBH}^\circ$ or belongs to one of the $\mathcal{H}'(k)$--factors). 
We also endow $\mathcal{B}'$ with decorations of Figure \ref{fig:brooms}, where $v'$ stands for the center of $\mathcal{B}'$, and $a'_n$, $b'_n$ correspond to $a_n$ and $b_n$, etc. 
Note that the set of points where $\mathcal{B}'$ is not locally path connected is given by $V'=J_{v'}-(\{a'_0\}\cup \{b'_0\})$ and therefore 
$u\in V'$. Since $u\not\in hf(\mathcal{WBH}^\circ)$ and also $u\neq v'$, without loss of generality, we assume $u\in J_{a'_0}-\{a'_0\}$. 
Further, continuity of $f$ implies $f(a_n)\to u$ and $f(b_n)\to u$ and for large $n$, both sequences $\{f(a_n)\}$ and $\{f(b_n)\}$ belong to a small neighborhood $U$ of $u$ consisting of infinitely many disjoint segments accumulating on  $J_{v'}\cap U$ (see Figure \ref{fig:brooms} for the illustration). 
Consider the shortest piece-wise linear paths $\alpha_n:I\longrightarrow \mathcal{B}'$, joining $\alpha_n(0)=f(a_n)\in U$ and $\alpha_n(1)=u$; $\beta_n:I\longrightarrow \mathcal{B}'$, joining $f(b_n)\in U$ and $u$. 
Clearly, both $\alpha_n$ and $\beta_n$ trace segments respectively:
\[
\alpha_n=[f(a_n),a'_0]\cup [u, a'_0]\subset \mathcal{B}',\qquad \beta_n=[f(b_n),a'_0]\cup [u, a'_0]\subset \mathcal{B}',
\] 
(we identify $\alpha_n$ and $\beta_n$ with their images for simplicity). 
In turn, the paths $g\circ \alpha_n$ and $g\circ \beta_n$, join points $\tilde{a}_n=g(f(a_n))$ and $v=g(f(u))$, see Equation \eqref{eq:in-J}. 
Since points $\tilde{a}_n$ (resp. $\tilde{b}_n$) are close to $a_n$ (resp. $b_n$), and belong to $J_{a_n}$ (resp. $J_{b_n}$).
The image of $g\circ \alpha_n$ contains $[\tilde{a}_n,a_0]\cup [v,a_0]$ and the image of  $g\circ \beta_n$ contains segments $[\tilde{b}_n,b_0]\cup [v,b_0]$. Therefore, for  $n$ large enough, we can find $s_n\in \alpha_n$, and $t_n\in \beta_n$, such that 
\begin{equation}\label{eq:s_n-t_n-to}
g(s_n)=a_0,\qquad g(t_n)=b_0.
\end{equation}
Moreover, for each $n$ we can choose minimal such $s_n$ and $t_n$ (i.e. closest to the initial point of the paths $\alpha_n$ and $\beta_n$).  
Passing to subsequences, if necessary, we have $s_n\to s$, $t_n\to t$, and both limits belong to $J_{a'_0}$. Clearly,
\begin{equation}\label{eq:g(s)-g(t)}
g(s)=a_0,\qquad g(t)=b_0.
\end{equation}
By \eqref{eq:s_n-t_n-to}, we have  $s\neq t$, and both $s$ and $t$ are above $u$, i.e. $s>u$ and $t>u$, according to the order defined after Equation \eqref{eq:B-union}. 

Suppose $s > t > u$: the initial points $f(a_n)$ of $\alpha_n$  converge to $u$, and  $s_n\in \alpha_n$ (or a subsequence) converges to $s$ as $n\to\infty$.
Since, $s> t$ we can find points $\{e_n\}$ with $e_n\in \alpha_n$ and
\[
 f(a_n)\leq e_n\leq s_n,
\]
(as points ordered along $\alpha_n$) and such that $e_n\to t$, as a consequence $g(e_n)\to g(t)=b_0$. However, $s_n$, is the first point on $\alpha_n$ mapped to $a_0$ under $g$. Further $f(a_n)$,  the initial point of $\alpha_n$, is mapped to $\tilde{a}_n$. By \eqref{eq:in-J}, we conclude that $g(e_n)\in J_{a_n}$ for large enough $n$, and therefore the limit of $\{g(e_n)\}$ has to belong to $J_{a_0}$, contradicting the fact that $b_0\not\in J_{a_0}$. 

In the case $t> s> u$, analogously considering paths $\beta_n$, we may find a sequence of points $\{h_n\}$, converging to $s$, and such that for large $n$:
\[
 f(b_n)\leq h_n\leq t_n.
\]
Then, again points $g(h_n)$ can only accumulate on $J_{b_0}$, contradicting 
\[
g(h_n)\to g(s)=a_0\not\in  J_{b_0}.
\]
This proves \eqref{eq:f(M(k))-in-hf} and concludes the proof of \eqref{eq:g-f-hf} for $\mathcal{WH}^\circ$ and $\mathcal{WBH}^\circ$. To prove the inclusion for the compactifications $\mathcal{WH}$ and $\mathcal{WBH}$, we see that Claim \eqref{eq:f(M(k))-in-hf} follows because the point at $\infty$, i.e $\{(0,0,0)\}$, is in the closure of $\bigcup_k M(k)$.  
\end{proof}
%
%
%%%%%%%%%%%%%%%%%%%%%%%%%%%%%%%%%%%%%%%%%%%%%%%%%%%%%%%%
\begin{proof}[Proof of Theorem \ref{thm:WH-WBH}]
Suppose that $\mathcal{WH}^\circ$ and  $\mathcal{WBH}^\circ$ are homotopy equivalent and take a homotopy equivalence $f:\mathcal{WH}^\circ\longrightarrow \mathcal{WBH}^\circ$, with inverse $g:\mathcal{WBH}^\circ\longrightarrow \mathcal{WH}^\circ$ (in the notation of Lemma \ref{lem:g-f-hf}.) Their restrictions to the sets of homotopically fixed points are defined by
\[
 \tilde{f}=f\bigl|_{hf(\mathcal{WH}^\circ)},\qquad \tilde{g}=g\bigl|_{hf(\mathcal{WBH}^\circ)}
\] 
By Lemma \ref{lem:g-f-hf}, compositions $\tilde{g}\circ\tilde{f}$ and $\tilde{f}\circ\tilde{g}$ are well defined and by Definition \ref{def:h.f.} they satisfy 
\[
 \tilde{g}\circ\tilde{f}=\text{id}_{hf(\mathcal{WH}^\circ)},\qquad  \tilde{f}\circ\tilde{g}=\text{id}_{hf(\mathcal{WBH}^\circ)}.
\]
Thus $\tilde{f}$, defines a homeomorphism between $hf(\mathcal{WH}^\circ)$ and $hf(\mathcal{WBH}^\circ)$, and $\tilde{g}$ its inverse, contradicting Lemma \ref{lem:hf(WH)-hf(WBH)}. For one-point compactifications:  $\mathcal{WH}$ and $\mathcal{WBH}$, points at infinity are homotopically fixed therefore the statement follows analogously.
\end{proof}
%
%
%%%%%%%%%%%%%%%%%%%%%%%%%%%%%%%%%%%%%%%%%%%%%%%%%%%
\subsection{Infinite wedges of products of \texorpdfstring{$n$}{n}--spheres}\label{sec:inf-wedges}
For arbitrarly high connected examples, we follow a similar pattern as in the previous section. Let $\mathcal{S}$ be an $n$--sphere, and $\mathcal{S}^2=\mathcal{S}\times \mathcal{S}$, define the following countable 
wedge products at a common basepoint $s$:
%%%%%%%%%%%%%%%%%%%%%%%%%%%%%%%%%%%%%%%%%%%%%%%%%%
\begin{equation}\label{eq:S_0-S_1}
  \mathcal{S}_0 =\mathcal{S}^2\vee\mathcal{S}^2\vee\cdots\vee\mathcal{S}^2\vee\cdots=\bigvee^\infty_{j=1} \mathcal{S}_0(j),\quad 
 \mathcal{S}_1 =\mathcal{S}\vee \mathcal{S}_0=\bigvee^\infty_{j=1} \mathcal{S}_1(j).
\end{equation}
Thus $\mathcal{S}_0$ and $\mathcal{S}_1$ differ just by the first factor, further we consider both $\mathcal{S}_0$ and $\mathcal{S}_1$ to be metrically embedded in $\R^{2 n+2}$, with the basepoint at the origin, and the diameters of factors $\mathcal{S}_\ast(j)$ tending to zero as $j\to\infty$, see Figure \ref{fig:S1-S0}. Both $\mathcal{S}_0$ and $\mathcal{S}_1$ are compact in the topology induced from the embedding. Clearly, $\mathcal{S}_0$ and $\mathcal{S}_1$ are $(n-1)$--connected in this topology since each factor is $(n-1)$--connected. Further, let $\mathcal{S}_\ast$ stand for either $\mathcal{S}_0$ or $\mathcal{S}_1$.

 The obvious retraction $r_{\mathcal{S}}:\mathcal{S}\times \mathcal{S}\longrightarrow \mathcal{S}$, can be extended by the identity (and rescailing) to a retraction $r_{\mathcal{S}_1}:\mathcal{S}_0\longrightarrow \mathcal{S}_1$. A retraction $r_{\mathcal{S}_0}:\mathcal{S}_1\longrightarrow \mathcal{S}_0$ can be obtained by simply collapsing the $\mathcal{S}$--factor of $\mathcal{S}_1$ to the basepoint. Thus we obtain 
\[
\mathcal{S}_0=_h \mathcal{S}_1,\qquad \mathcal{S}_0=_h \mathcal{S}_1.
\]
%
%
%%%%%%%%%%%%%%%%%%%%%%%%%%%%%%%%%%%%%
\begin{figure}[!ht] 
  \centering
   \includegraphics[width=0.3\textwidth]{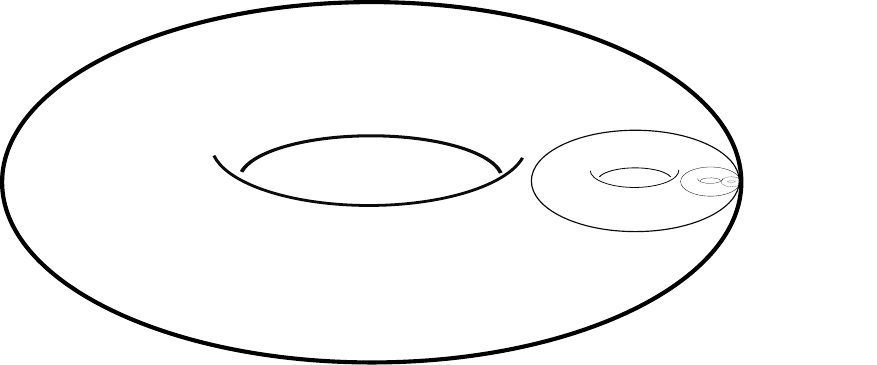}\quad \includegraphics[width=0.3\textwidth]{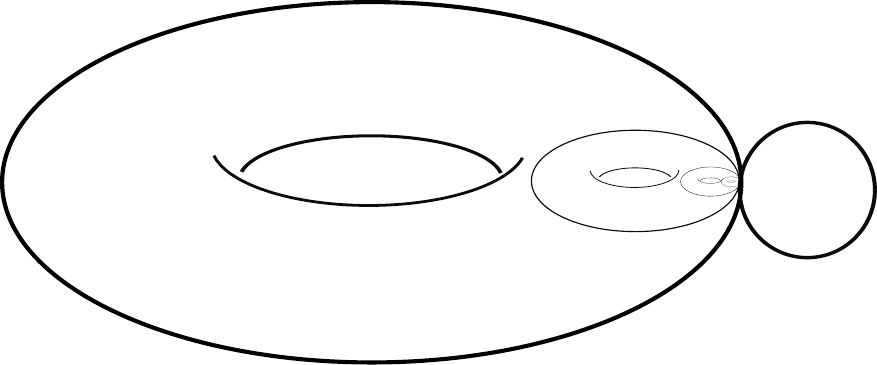}
  \caption{Hawaiian earrings $\mathcal{S}_0$(left) and $\mathcal{S}_1$(right) based on $\mathcal{S}\times \mathcal{S}$ and $\mathcal{S}$, for $n=1$.} \label{fig:S1-S0} 
\end{figure} 
%%%%%%%%%%%%%%%%%%%%%%%%%%%%%%%%%%%%%%%%%%%%%%%%%%%%%%%%%%%%%%%%%%%%%%%%%%%%%%%%%%%%%%%
The strategy for proving that $\mathcal{S}_0$ and $\mathcal{S}_1$ are not homotopy equivalent is a little different than before and based on 
some homological considerations. 

Both $\mathcal{S}_0$ and $\mathcal{S}_1$ are a special case of the generalized Hawaiian earrings construction considered in \cite{Eda-Kawamura00}.  Following \cite{Eda-Kawamura00} consider the following homomorphisms defined on $\pi_\ast(\mathcal{S}_\ast,s)$ in dimension $n$ (for $n>1$): 
\[
 h_0:\pi_n(\mathcal{S}_0,s)\longrightarrow \prod^\infty_{j=1} \pi_n(\mathcal{S}_0(j),s),\qquad h_1:\pi_n(\mathcal{S}_1,s)\longrightarrow \prod^\infty_{j=1} \pi_n(\mathcal{S}_1(j),s),
\]
and induced by the product of obvious coordinate retractions $r_{j,\ast}:\mathcal{S}_\ast\longrightarrow \mathcal{S}_\ast(j)$ onto each factor of $\mathcal{S}_\ast$. The main theorem of \cite[p. 18]{Eda-Kawamura00} implies that both $h_0$ and $h_1$ are isomorphisms. 
By the Hurewicz Theorem 
\begin{equation}\label{eq:H_n-S_ast}
\begin{split}
 H_n(\mathcal{S}_0) & \cong \pi_n(\mathcal{S}_0,s)\cong \prod^\infty_{j=1} H_n(\mathcal{S}_0(j);\Z)\cong \prod^\infty_{j=1} (\Z\times \Z)=(\Z\times \Z)^\omega\cong \Z^\omega,\\
 H_n(\mathcal{S}_1) & \cong \pi_n(\mathcal{S}_1,s)\cong H_n(\mathcal{S};\Z)\times \prod^\infty_{j=2} H_n(\mathcal{S}_1(j);\Z)\cong \Z\times (\Z\times \Z)^\omega\cong \Z^\omega.
\end{split}
\end{equation}
\no Since both $S_0$ and $S_1$ are $(n-1)$--connected, for $n>1$, the Universal Coefficients Theorem for cohomology and \eqref{eq:H_n-S_ast} implies additively
\begin{equation}\label{eq:H^n-S_ast}
 H^n(\mathcal{S}_\ast)\cong \text{Hom}(H_n(\mathcal{S}_\ast;\Z);\Z)\cong \bigoplus^\infty_{j=1} H^n(\mathcal{S}_\ast(j);\Z)\cong  \bigoplus^\infty_{k=1}\Z,
\end{equation}
(c.f. \cite[p. 67]{Nunke62} for the second isomorphism). 
%%%%%%%%%%%%%%%%%%%%%%%%%%%%%%%%%%%%%%%%%%%%%%%%%%%
\begin{theorem}\label{thm:S_0-S_1}
$\mathcal{S}_0$ and $\mathcal{S}_1$ are $h$--equal but not homotopy equivalent.
\end{theorem} 
\begin{proof}
The $h$--equality has been already argued at the beginning of this section. For the second claim, first we note that the graded ring structures of each factor $H^\ast(\mathcal{S}_\ast(j);\Z)$ are well known, i.e.
\begin{equation}\label{eq:H^ast-factors}
\begin{split}
 H^\ast(\mathcal{S}_0(j)) & \cong \Z[x_j,y_j]\Bigl/\langle x^2_j=0, y^2_j=0\rangle,\\
 H^\ast(\mathcal{S}_1(1)) & \cong \Z[w]\Bigl/\langle w^2=0\rangle,\ H^\ast(\mathcal{S}_1(k))  \cong \Z[x_k,y_k]\Bigl/\langle x^2_k=0, y^2_k=0\rangle,\ k>1,
\end{split}
\end{equation}
where $x_i$, $y_i$ and $w$ are of degree $n$. Observe that the graded ring $H^\ast(\mathcal{S}_0)$ has the following property: 
%%%%%%%%%%%%%%%%%%%%%%% property H^n
\begin{quote}
$(\ast)$ {\em For any nontrival  $p$ in $H^n(\mathcal{S}_0)$  there exists $q$  in $H^n(\mathcal{S}_0)$ such that $p\cdot q\neq 0$.}
\end{quote}
\no Indeed, from \eqref{eq:H^n-S_ast} any $p\in H^n(\mathcal{S}_0)$ is given as 
\[
 p=\sum^\infty_{i=1} (a_i x_i+b_i y_i),\qquad a_i,b_i\in \Z,
\]
where only finitely many $a_i$'s and $b_i$'s are nonzero. Let $r:\mathcal{S}_0\longrightarrow \mathcal{S}_0(1)\vee\cdots \vee \mathcal{S}_0(k)$ be a retraction on $k$ first factors of $\mathcal{S}_0$, and $r^\ast:H^\ast(\bigvee^k_{l=1} \mathcal{S}_0(l))\longrightarrow H^\ast(\mathcal{S}_0)$ the induced monomorphism. Choosing $k$ large enough, and using the same symbols for the generators of $H^n(\bigvee^k_{l=1} \mathcal{S}_0(l))$ as in \eqref{eq:H^ast-factors}, we have
\[
 p=r^\ast(p')=r^\ast(\sum^k_{i=1} (a_i x_i+b_i y_i)), 
\]
for some $p'\in H^n(\bigvee^k_{l=1} \mathcal{S}_0(l))$. Note that $H^\ast(\bigvee^k_{l=1} \mathcal{S}_0(l))\cong \bigoplus^k_{l=1} H^\ast(\mathcal{S}_0(l))$ as graded rings.
Let $q'=y_i$, for $i$ such that $a_i\neq 0$, then $p'\cdot q'=a_i x_i y_i\neq 0$  and we obtain
\[
p\cdot q=r^\ast(p'\cdot q')\neq 0,
\]
by the injectivity of $r^\ast$. Clearly, the property $(\ast)$ is preserved under the graded ring isomorphisms.
 Note that for $H^\ast(\mathcal{S}_1)$ the property $(\ast)$ does not hold.
Indeed, let $p=w\in H^n(\mathcal{S}_1)$, which generates the cohomology of the $\mathcal{S}$--factor of $\mathcal{S}_1$, \eqref{eq:H^ast-factors}. By \eqref{eq:H^n-S_ast}, any $q\in H^n(\mathcal{S}_1)$ is represented by
\[
 q=c_0 w+\sum^\infty_{i=1} (a_i x_i+b_i y_i),\qquad c_0,a_i,b_i\in \Z,
\]
where only finitely many $a_i$'s and $b_i$'s are nonzero. Again, choosing an appropriate retraction $r$ of $\mathcal{S}_1$ onto finitely many factors,  we have $p=r^\ast(p')$, $q=r^\ast(q')$ for some $p'$ and $q'$ in $H^n(\bigvee^k_{l=1} \mathcal{S}_1(l))$, and therefore 
\[
 p\cdot q=r^\ast(p'\cdot q')=r^\ast(w\cdot q')=0.
\]
We conclude, as graded rings 
\[
 H^\ast(\mathcal{S}_0)\not\cong H^\ast(\mathcal{S}_1),
\]
and consequently $\mathcal{S}_0$ and $\mathcal{S}_1$ cannot be homotopy equivalent. 
\end{proof} 

 Modifying slightly the construction of $\mathcal{S}_0$ and $\mathcal{S}_1$, one may consider an infinite compact bouquets $\mathcal{WS}_0$ and $\mathcal{WS}_1$ embedded in $\R^{2n+2}$ as pictured on Figure \ref{fig:S1-S0-long}, having the same factors as $\mathcal{S}_0$ and $\mathcal{S}_1$. We claim that $\mathcal{WS}_0$ and $\mathcal{WS}_1$ are homotopy equivalent to their respective counterparts $\mathcal{S}_0$ and $\mathcal{S}_1$. For this purpose, let us consider just the case of  $\mathcal{WS}_0$ and $\mathcal{S}_0$  (for $\mathcal{WS}_1$ and $\mathcal{S}_1$ the claim follows analogously).
 %%%%%%%%%%%%%%%%%%%%%%%%%%%%%%%%%%%%%
 \begin{figure}[!ht] 
 	\centering
 	\includegraphics[width=0.45\textwidth]{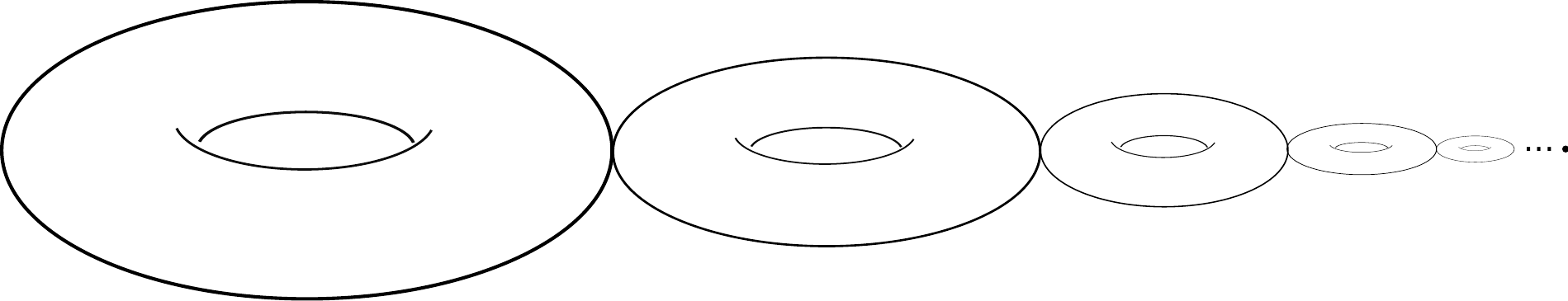}\qquad \includegraphics[width=0.45\textwidth]{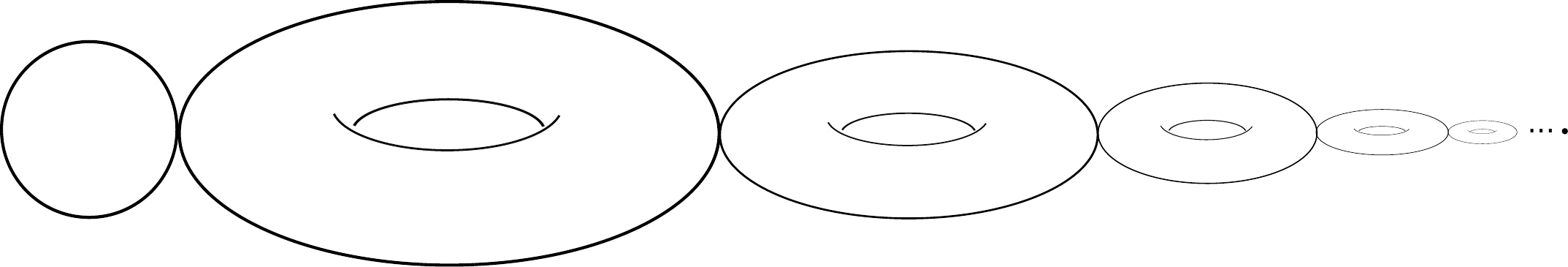}
 	\caption{Bouquets $\mathcal{WS}_0$(left) and $\mathcal{WS}_1$(right), for $n=1$.} \label{fig:S1-S0-long} 
 \end{figure}  
For simplicity, we just consider the case $\mathcal{S}$ is a circle ($n=1$), then all spaces can be embedded in $\R^3$ as shown\footnote{Note that Figure \ref{fig:S1-S0} shows a different embedding of $\mathcal{S}_0$ in $\R^3$.} on Figure \ref{fig:S0-WS0-BS0-homotopy}. The following construction can be conducted for any $n$, without major changes. Before we describe the intermediate space $\mathcal{MS}_0$ of Figure \ref{fig:S0-WS0-BS0-homotopy}, let us define the map 
\begin{equation}\label{eq:Q-map}
Q:\mathcal{WS}_0\longrightarrow \mathcal{S}_0,
\end{equation}
which is the required homotopy equivalence.  
 Denote by $a_i\subset \mathcal{WS}_0$ arcs connecting successive wedge points in $\mathcal{WS}_0$, and  $A=\bigcup_i a_i$ their union. The map $Q$ is the quotient projection identifying $A$ with the wedge point $x_0$ of $\mathcal{S}_0$. 

Next, let us define an intermediate space $\mathcal{MS}_0$. It is built from a contractible broom $\mathcal{M}=\bigcup^\infty_{i=0} c_i$ of a carefully chosen sequence of segments $\{c_i\}^\infty_{i=0}$ in $\R^3$, all having the common right endpoint, i.e. the wedge point $x_0$ of $\mathcal{M}$. At the left endpoints of $c_i$'s we attach $\mathcal{S}^2$--factors, as shown on Figure \ref{fig:S0-WS0-BS0-homotopy}. 
 We assume that the sequence $\{c_i\}^\infty_{i=0}$ starts with a horizontal segment $c_0=[-1,0]$ contained in the $x$--axis and all $c_i$ are placed above $c_0$ within the $xz$--plane of $\R^3$. 
As $i\to \infty$, we require the lengths of $c_i$'s, and the diameters of $\mathcal{S}^2$--factors tend to zero. %Also the angle between $c_i$ and $c_0$ ought to tend to a fixed angle between $0$ and $\frac{\pi}{2}$.
As a result, $c_i$'s and the $\mathcal{S}^2$--factors of $\mathcal{MS}_0$ are converging to the point $x_0\in \mathcal{MS}_0$ in the limit. The map 
 \begin{equation}\label{eq:C-map}
 C:\mathcal{S}_0\longrightarrow \mathcal{MS}_0,
 \end{equation}
in Figure \ref{fig:S0-WS0-BS0-homotopy} is defined as follows; let $c_i$ be a small contractible neighborhood of $x_0$ in the $i$th $\mathcal{S}^2$--factor of $\mathcal{S}_0$. The map $C$ collapses each $c_i$ to the corresponding segment in $\mathcal{MS}_0$, also called $c_i$. Clearly, $\mathcal{S}_0$ and $\mathcal{MS}_0$ are homotopy equivalent under $C$, the inverse homotopy equivalence simply contracts the broom $\mathcal{M}$ to the wedge point $x_0$ in $\mathcal{MS}_0$.
% 
% 
% 
%%%%%%%%%%%%%%%%%%%%%%%%%%%%%%%%%%%%%
\begin{figure}[!ht] 
	\centering
	\begin{overpic}[width=.97\textwidth]{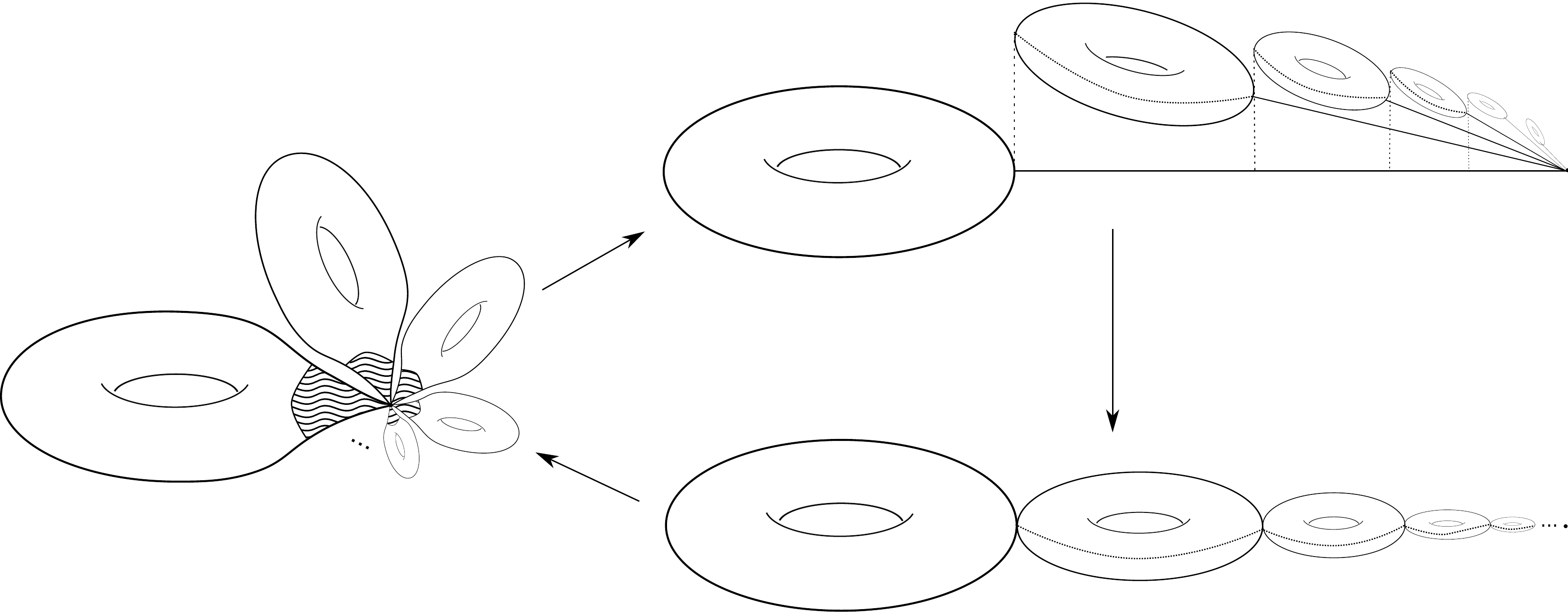}
		% names of spaces
		\put(12,22){$\mathcal{S}_0$}
		\put(80,10){$\mathcal{WS}_0$}
		\put(56,36){$\mathcal{MS}_0$}
		
		% cones
		\put(16.5,13){$c_0$}
		\put(22.5,17.8){{\small $c_1$}}
		\put(26.5,16){{\tiny $c_2$}}
		% arcs
		\put(74,4){\tiny{$a_1$}}
		\put(87,4.7){{\tiny $a_2$}}
		\put(95,4){{\tiny $a_3$ \ldots}}		
		%\put(75,-1){$A=a_1\cup a_2\cup\ldots$}
		% Q
		\put(38,10){$Q$}
		% C
		\put(38,20){$C$}
		% P
		\put(73,17){$P$}
		% b arcs
		\put(75,26){$c_0$}
		\put(72,33){{\tiny $a_1$}}
		\put(86,29.5){{\tiny $c_1$}}
		\put(85,33){{\tiny $a_2$}}		
		\put(90,30.5){{\tiny $c_2$}}
		% x_0 in MS_0
		\put(100,27){{\tiny $x_0$}}
		% x_0 in WS_0
		\put(100,5.7){{\tiny $x_0$}}				
		% x_0 in S_0
		\put(22.5,11.3){{\tiny $x_0$}}
	\end{overpic} 
	\caption{Homotopy equivalences $P$, $Q$ and $C$ on spaces $\mathcal{S}_0$, $\mathcal{WS}_0$ and $\mathcal{MS}_0$.} \label{fig:S0-WS0-BS0-homotopy} 
\end{figure}
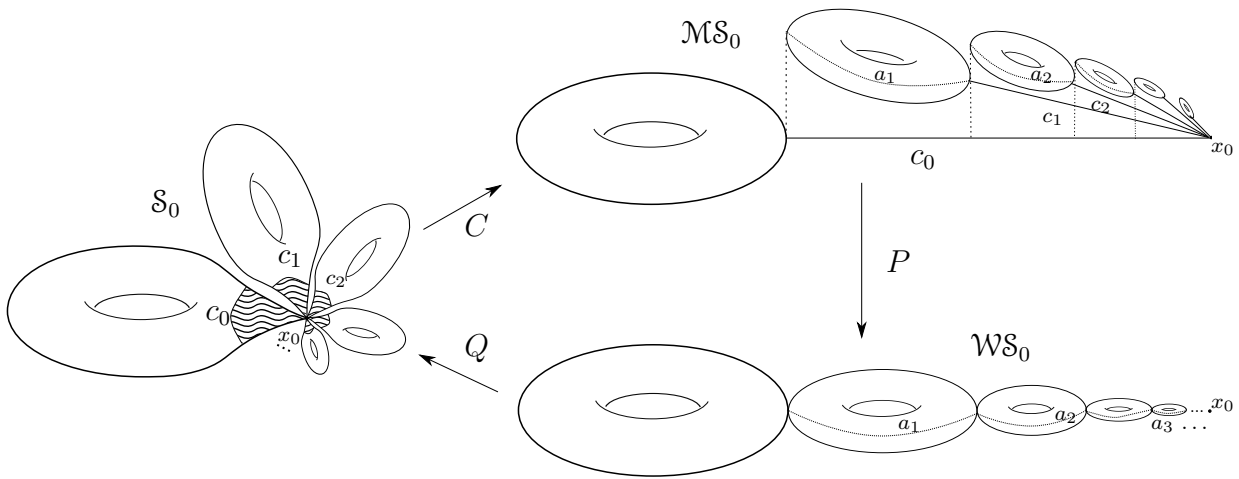  
\no Now, let  $\mathcal{M}'$ be an extension of the broom $\mathcal{M}$ in $\mathcal{MS}_0$, given by 
\[
\mathcal{M}'=c_0\cup (a_1\cup c_1)\cup (a_2\cup c_2)\cup \ldots = c_0\cup \bigcup^\infty_{i=1}\, (a_i\cup c_i), 
\]
where each $c_i$, for $i>0$, is extended by adjoining an arc $a_i$ to the left endpoint of $c_i$. In order to define the map
\begin{equation}\label{eq:P-map}
P:\mathcal{MS}_0\longrightarrow\mathcal{WS}_0,
\end{equation}
of Figure \ref{fig:S0-WS0-BS0-homotopy}, first, consider a retraction $p:\mathcal{M}'\longrightarrow c_0$ of $\mathcal{M}'$ onto the $c_0$ segment of $\mathcal{M}'$.  The arcs $a_i\cup c_i$ can be arranged appropriately, so that  $p$ can be induced by restriction of the projection $(x,y,z)\longrightarrow (x,0,0)$ in $\R^3$ to $\mathcal{M}'$, mapping each $a_i\cup c_i$ of the broom onto a portion of $c_0$. Clearly,  $p$ defines a deformation retraction\footnote{via e.g. a piecewise linear homotopy}  of $\mathcal{M}'$ onto $c_0$. The map $P$  identifies the segment $c_0\subset \mathcal{MS}_0$ with the arc $A=\bigcup_i a_i$ in $\mathcal{WS}_0$. The map $P$ is equal to $p$ along $\mathcal{M}'\subset \mathcal{MS}_0$, and identifies the $\mathcal{S}^2$--factors of $\mathcal{MS}_0$ with the corresponding factors in $\mathcal{WS}_0$. 
  
\begin{proposition}
	The Hawaiian earrings $\mathcal{S}_0$ and $\mathcal{S}_1$ of Figure \ref{fig:S1-S0} are homotopy equivalent to infinite wedges $\mathcal{WS}_0$ and $\mathcal{WS}_1$ of Figure \ref{fig:S1-S0-long}.
\end{proposition}
\begin{proof}[Sketch of Proof]
	Having the maps $C$ and $P$ defined, we set $\widehat{Q}=P\circ C$, and claim that it defines a homotopy inverse of $Q$. Indeed, the composition $Q\circ\widehat{Q}:\mathcal{S}_0\longrightarrow \mathcal{S}_0$ maps the inverse image $C^{-1}(\mathcal{M})$ to the point $x_0$, and is an embedding on $\mathcal{S}_0-C^{-1}(\mathcal{M})$. The set $C^{-1}(\mathcal{M})$ is the union of $c_i\subset \mathcal{S}_0$ and arcs $C^{-1}(a_i)$, thus $C^{-1}(\mathcal{M})$ can be continuously contracted to the point $x_0\in \mathcal{S}_0$, implying $Q\circ\widehat{Q}\simeq \text{id}_{\mathcal{S}_0}$. The composite $\widehat{Q}\circ Q:\mathcal{WS}_0\longrightarrow\mathcal{WS}_0$, maps small contractible neighborhoods $U(a_i)$ of arcs $a_i$ in $\mathcal{WS}_0$ together with paths\footnote{$Q$ is an embedding away from the arc $A$} $g_i=Q^{-1}(C^{-1}(a_i))$ to the arc $A=\bigcup_i a_i$. Both $A$ and the union of subsets $U(a_i)\cup g_i$ are contractible within $\mathcal{WS}_0$, implying again $\widehat{Q}\circ Q\simeq \text{id}_{\mathcal{WS}_0}$.
\end{proof}
\begin{remark}
Observe that, contrary to $\mathcal{S}_\ast$, the complement of the infinity point in $\mathcal{WS}_\ast$, denoted by  $\mathcal{WS}^\circ_\ast$ is an $(n-1)$--connected, locally contractible, ANR. Note that if $\mathcal{WS}_\ast$ were locally $2n$--connected, it would imply it was an ANR. In this sense, $\mathcal{WS}_0$ and $\mathcal{WS}_1$ are ``close'' to being ANR spaces. In \cite{Stewart58}, Stewart shows  $\mathcal{WS}^\circ_0$ and $\mathcal{WS}^\circ_1$ for $n=1$ are examples of noncompact  $h$--equal ANRs, which are not homotopy equivalent. The proof of \cite{Stewart58}, relies heavily on the non-triviality of the fundamental group $\mathcal{S}_0$ and $\mathcal{S}_1$ and on the structure of isomorphisms between free products of groups, \cite{Stewart58}. We wish to point out that the computation of cohomology rings $H^\ast(\mathcal{WS}^\circ_0)$, $H^\ast(\mathcal{WS}^\circ_1)$ is basic and the argument of Theorem \ref{thm:S_0-S_1} can be easily adapted to show $H^\ast(\mathcal{WS}^\circ_0)\not\cong H^\ast(\mathcal{WS}^\circ_1)$, which implies  $\mathcal{WS}^\circ_0\not\simeq \mathcal{WS}^\circ_1$. Moreover, the argument is valid for all $n\geq 1$, giving $2n$--dimensional examples which are $(n-1)$--connected for $n\geq 1$. 

 In examples, from the last two sections, homotopy dominations are given by retractions. Consequently, these are examples of $r$--equal continua, which are not homotopy equivalent (c.f. \cite{Borsuk67}). 
\end{remark}

%
%
%%%%%%%%%%%%%%%%%%%%%%%%%%%%%%%%%%%%%%%%%%%%%%%%%%%%%%%%%%%%%%%%%%%%%%%%%%%%%

\end{document}